\documentclass[10pt]{amsart}
\usepackage{enumerate}

\usepackage{color}
\usepackage{palatino}
\usepackage{hyperref}
\usepackage{url}

\usepackage[nice]{nicefrac}

\renewcommand{\baselinestretch}{1.0}
\usepackage{amscd}
\usepackage{amsthm}
\usepackage{amsxtra}
\usepackage{amsfonts,amssymb,amsmath,mathrsfs,graphicx}
\usepackage{float}
\usepackage{epsfig}
\usepackage{graphicx}
\usepackage{verbatim}
\renewcommand{\baselinestretch}{\baselinestretch}
\renewcommand{\baselinestretch}{1.2}
\baselineskip=16pt \hypersetup{
 pdftitle={Translating solitons to the ${\mathcal{K}}^{\frac{1}{4}}$-flow in
 ${\mathbb{R}}^3$}, 
pdfauthor={Hojoo Lee}, 
pdfnewwindow=true, 
colorlinks=true, 
linkcolor=black, 
citecolor=blue, 
filecolor=magenta, 
urlcolor=blue 
}
\newtheorem{theorem}{Theorem}

\newtheorem{lemma}[theorem]{Lemma}

\theoremstyle{definition}

\theoremstyle{remark}

\numberwithin{equation}{section}

\begin{document}


\title[Isometric deformations of the ${\mathcal{K}}^{\frac{1}{4}}$-flow translators in ${\mathbb{R}}^3$ with helicoidal symmetry]
 {Isometric deformations of the ${\mathcal{K}}^{\frac{1}{4}}$-flow translators in ${\mathbb{R}}^3$ with helicoidal symmetry}
\author{Hojoo Lee}
\address{Department of Geometry and Topology, University of Granada, Granada, Spain.}
\email{ultrametric@gmail.com}
 \thanks{This work was supported by the National Research Foundation of Korea Grant funded
 by the Korean Government (Ministry of Education, Science and Technology) [NRF-2011-357-C00007].}

\maketitle
\begin{abstract}
 The height functions of ${\mathcal{K}}^{\frac{1}{4}}$-flow translators in  Euclidean space ${\mathbb{R}}^3$ solve
 the classical Monge-Amp\`{e}re equation $f_{xx} f_{yy} - {f_{xy}}^2 = 1$. We explicitly and geometrically determine the
 moduli space of all helicoidal ${\mathcal{K}}^{\frac{1}{4}}$-flow translators, which are generated from planar curves
 by the action of helicoidal groups.
\end{abstract}

\section{Motivation and main results}

\subsection{Introduction}
 The classical curve-shortening flow adimts fruitful generalizations with intriguing
 applications. One of Huisken's theorems guarantees that an analogue of the Gage-Hamilon's
 shrinking-curves theorem in the plane also holds for the mean curvature flow
 in higher dimensional Euclidean spaces.

 Andrews \cite{An99} proved Firey's conjecture that convex surfaces evolving
 by the Gauss curvature flow become spherical.  Chow \cite{Ch85} investigated the normal deformation by powers of the Gauss curvature, and
 Urbas \cite{Ur98} studied self-similar and translating solitons for the normal evolution by positive powers of the Gauss curvature.

 We say that a surface $\Sigma$ is a \textit{${\mathcal{K}}^{\frac{1}{4}}$-translator} when we have the geometric condition
\begin{equation*}
     {\mathcal{K}}_{{}_{\Sigma}}  =  \cos^{4}\left({\theta}_{{}_{\Sigma}}\right).
\end{equation*}
 The scalar function ${\mathcal{K}}_{{}_{\Sigma}}$ denotes the Gaussian curvature and the third component $\cos  \left({\theta}_{{}_{\Sigma}}\right) = {\mathbf{n}}_{{}_{\Sigma}} \cdot {(0,0,1)}$ of the unit
 normal ${\mathbf{n}}_{{}_{\Sigma}}$ is called the angle function on $\Sigma$.

  The ${\mathcal{K}}^{\frac{1}{4}}$-translators in Euclidean space ${\mathbb{R}}^3$ are of significant geometrical
  interest. The convex graph $z=f(x,y)$ becomes a ${\mathcal{K}}^{\frac{1}{4}}$-translator if and only if
  its height function $f$ solves the classical Monge-Amp\`{e}re equation
\begin{equation*}
  f_{xx} f_{yy} - {f_{xy}}^2 = 1.
\end{equation*}

 J\"{o}rgens' outstanding holomorphic resolution \cite{Jo54} says that, when $f_{xx} f_{yy} - {f_{xy}}^2 = 1$,
 the gradient graph $\left(x,y,f_{x},f_{y}\right)$ becomes a minimal surface in Euclidean space  ${\mathbb{R}}^4$. The
 Hessian one equation is a special case of special Lagrangian equations \cite{HL82}, \textit{split} special Lagrangian equations \cite{HL10, Le12, Me91},
 and affine mean curvature equations \cite{An96, Ca82, TW05}. Furthermore, its
 solutions induce flat surfaces in  hyperbolic space ${\mathbb{H}}^{3}$ \cite{Sp79}.

 \subsection{Isometric deformations of helicoidal ${\mathcal{K}}^{\frac{1}{4}}$-translators}

\begin{theorem}[\textbf{Moduli space of ${\mathcal{K}}^{\frac{1}{4}}$-translators with rotational
\& helicoidal symmetry}] \label{moduli0}
 \textbf{(A)}  Any helicoidal ${\mathcal{K}}^{\frac{1}{4}}$-translator $\Sigma$ of pitch $\mu$ admits
 a one-parameter family of isometric helicoidal ${\mathcal{K}}^{\frac{1}{4}}$-translators ${\Sigma}^{h}$
 with pitch $h$ such that $\Sigma={\Sigma}^{\mu}$ and that ${\Sigma}^{0}$ is rotational. \\
 \textbf{(B)} The cylinder over a circle in the $xy$-plane is a
 rotational ${\mathcal{K}}^{\frac{1}{4}}$-translator. Additionally, there exists a one-parameter family of
 ${\mathcal{K}}^{\frac{1}{4}}$-translators ${\mathcal{H}}_{c}$ invariant under the rotation with
 $z$-axis. The profile curve of rotational surface ${\mathcal{H}}_{c}$ is
 congruent to the graph $( \; U,  \; 0,  \; {\Lambda}_{c}(U)  \;)$, where the one-parameter family of
 height functions ${\Lambda}_{c}(U)$ is explicitly given by
  \begin{equation*}
  {\Lambda}_{c}(U) =
\begin{cases}
   \frac{1}{2} \left[ \, U \sqrt{U^2 + {\kappa}^{2}\,}  + {\kappa}^{2} \, \mathrm{arcsinh} \,
 \left( \frac{U}{{\kappa}}  \right)  \; \right],  \quad \quad U>0 \; (\text{when}\; c=1 + {\kappa}^{2}, \; \kappa>0),  \\
   \frac{1}{2} \; U^2, \quad \quad \quad \quad  \quad \quad  \quad \quad \quad \quad  \quad \quad \quad \quad
   \quad U \geq 0 \; (\text{when}\; c=1),  \\
   \frac{1}{2} \left[ \,
 U \sqrt{U^2 - {\kappa}^{2}\,} -  {\kappa}^{2} \, \mathrm{arccosh} \,  \left( \frac{U}{{\kappa}}  \right)
  \; \right],   \quad \quad U > \kappa \; (\text{when}\; c=1 - {\kappa}^{2}, \; \kappa>0).
\end{cases}
\end{equation*}
 \textbf{(C)} There exists a two-parameter family of helicoidal ${\mathcal{K}}^{\frac{1}{4}}$-translators
 ${\mathcal{H}}^{h}_{c}$ and the geometric coordinates  $(U,t)$ on ${\mathcal{H}}^{h}_{c}$ satisfying the
 following conditions.
 \begin{enumerate}
  \item[\textbf{(C1)}] The geometric meaning of the parameter $h$ is that the surface  ${\mathcal{H}}^{h}_{c}$ is
  invariant under the helicoidal motion with pitch $h$. The surface ${\mathcal{H}}^{h}_{c}$ is invariant under the one-parameter subgroup
  $\{ {\mathbf{S}}_{T} \}$ of the group of rigid motions of ${\mathbb{R}}^{3}={\mathbb{C}} \times {\mathbb{R}}$
  given by
\begin{equation*}
   \left( \, \zeta, \, z  \, \right) \in {\mathbb{C}} \times {\mathbb{R}} \; \mapsto
   \; {\mathbf{S}}_{T} \left(  \, \zeta, \, z  \, \right) =  \left( \, e^{iT} \zeta,  \, hT + z  \, \right) \in {\mathbb{C}} \times {\mathbb{R}}.
\end{equation*}
 \item[\textbf{(C2)}]  There exist the coordinates $(U,t)$ on the helicoidal surface ${\mathcal{H}}^{h}_{c}$ such that
 its metric reads $I_{{\mathcal{H}}^{h}_{c}} = \left(U^2 + c\right) dU^2 +U^2 dt^2$.
 \item[\textbf{(C3)}] The geometric meaning of the parameter $c$ is the property that the
 helicoidal surface  ${\mathcal{H}}^{h}_{c}$ is isometric to the rotational surface
 ${\mathcal{H}}^{0}_{c}={\mathcal{H}}_{c}\,$.
 \item[\textbf{(C4)}] The geometric meaning of the coordinate $U$ is the property that the function $\frac{1}{\sqrt{U^2 +c}}$ coincides with the angle
 function on the surface ${\mathcal{H}}^{h}_{c}$ up to a sign.
\end{enumerate}
\end{theorem}

 The statement \textbf{(A)} in Theorem \ref{moduli0} is inspired by the 1982 do Carmo-Dajczer theorem \cite{CD82} that a surface of non-zero constant mean
 curvature is helicoidal if and only if it lies in the  associate family \cite{La70} of a Delaunay's
 rotational surface \cite{Ee87, Ko00}  with the same constant mean curvature. In 1998, Haak \cite{Ha98} presented an alternative
 proof of the do Carmo-Dajczer theorem.

 The \textit{mean curvature flow} in ${\mathbb{R}}^3$ also admits the translating solitons with
 helicoidal symmetry. In 1994, Altschuler and Wu \cite{AW94} showed the existence of the convex, rotational,
 entire graphical translator. In 2007, Clutterbuck, Schn\"{u}rer and Schulze \cite{CSS07} constructed the bigraphical translator,
 which is also rotationally symmetric.

  \textbf{Open Problem.}  Prove or disprove that Halldorsson's helicoidal translators \cite{Ha11} for the mean curvature
 flow admit the isometric deformation from rotational translators.

\bigskip

 \textbf{Acknowledgement.}  I would like to thank Miyuki Koiso for sending me the paper \cite{Ko00}
 and appreciate discussions with Matthias Weber.

\bigskip

\section{Proof of Theorem \ref{moduli0}}   \label{Sheli}

   We first need to revisit Bour's construction \cite{CD82} with details to specify the behavior
 of the angle function on his isometric helicoidal surfaces.

\begin{lemma}[\textbf{Angle function on Bour's helicoidal surfaces}] \label{B00}
 Let $\Sigma$ be a helicoidal surface with pitch vector
 $\mu \mathbf{k}=(0,0,\mu)$ and the generating curve $\gamma=({\mathcal{R}}, 0, \Lambda)$ in
 the $xz$-plane, which admits the parametrization
$  \left(u, \theta\right) \mapsto   \left( \, {\mathcal{R}} \cos
\theta, \, {\mathcal{R}} \sin \theta, \, \Lambda  + \mu \theta  \,
\right)$,
 where $u$ denotes a parameter of the generating curve  $\gamma$. We then define the Bour coordinate transformation
\begin{equation*}
   \left(u, \theta\right) \mapsto   \left( s, t \right) =  \left( s, \theta + \Theta \right),
\end{equation*}
via the relations
\begin{equation*} \label{data1}
\begin{cases}
   ds^2 = d{\mathcal{R}}^2 + \frac{{\mathcal{R}}^2}{{\mathcal{R}}^2 + {\mu}^2 } d\Lambda^{2}, \\
   d\Theta =  \frac{\mu}{{\mathcal{R}}^2 + {\mu}^2 } d{\Lambda},
\end{cases}
\end{equation*}
 and also introduce the Bour function $U$ using the relation $U^2 =R^2 +  {\mu}^2$.  \\
 \textrm{\textbf{(A)}} The  helicoidal surface $\Sigma$ admits the reparametrization satisfying  \textrm{\textbf{(A1)}}, \textrm{\textbf{(A2)}}, and
 \textrm{\textbf{(A3)}}:
\begin{equation*}
   \left( s, t \right) \mapsto  \mathbf{X} \left( s, t \right) = \left(\, {\mathcal{R}} \cos  \left( t - \Theta \right), \, {\mathcal{R}}  \left( t - \Theta \right), \,
   \Lambda + \mu  \left( t - \Theta \right) \, \right).
\end{equation*}
 \textbf{(A1)} Its first fundamental form reads $I_{\Sigma} = ds^2 + U^2 dt^2$. \\
 \textbf{(A2)} The parameters $R$, $\Lambda$, and $\Theta$ can be recovered from the Bour function $U$ explicitly:
\begin{equation*}
\begin{cases}
   {\mathcal{R}}^2 = U^2 -  {\mu}^2, \\
    d\Lambda ^2 = \frac{ U ^2 }{\left(  U^2 -  {\mu}^2  \right)^2}
    \left(  U^2 \, \left( 1 - \left(  \frac{ dU   }{ ds } \right)^2  \right) - h^2
    \right) \, ds^2, \\
    d  \Theta = \frac{\mu}{U^2  } d\Lambda.
\end{cases}
\end{equation*}
 \textbf{(A3)} The angle function $n_3$ defined as the third component  $\mathbf{n} \cdot \mathbf{k}$ of the induced unit
 normal $\mathbf{n}=\frac{1}{ \Vert {\mathbf{X}}_{s} \times {\mathbf{X}}_{t} \Vert} {\mathbf{X}}_{s} \times
 {\mathbf{X}}_{t}$ is also determined by the Bour function $U$.
 \begin{equation*}
 {n_3}^2 =  \left(  \frac{dU}{ds} \right)^{2}.
\end{equation*}
 \textbf{(B)} We construct a two-parameter family of
 helicoidal surfaces ${\Sigma}^{\lambda, h}$ of pitch $h$ by the
 patch
\begin{equation*}
   {\mathbf{X}}^{\lambda, h} \left( s, t \right)
   = \left(\, {{\mathcal{R}}}^{\lambda, h} \cos  \left( \frac{t}{\lambda} - {\Theta}^{\lambda, h} \right),
   \, {{\mathcal{R}}}^{\lambda, h} \left(
   \frac{t}{\lambda} -  {\Theta}^{\lambda, h}  \right), \, {\Lambda}^{\lambda, h}  + h \left( \frac{t}{\lambda} -
   {\Theta}^{\lambda, h}  \right) \, \right),
\end{equation*}
where the geometric datum $\left({{\mathcal{R}}}^{\lambda, h},
{\Lambda}^{\lambda, h}, {\Theta}^{\lambda, h}  \right)$ is
explicitly determined by the pair $(\lambda, h)$ of constants  and
the Bour function $U(s)$ arising from the reparametrization
$\mathbf{X} \left( s, t \right)$ of  $\Sigma$
\begin{equation} \label{data2}
\begin{cases}
   {\left( {{\mathcal{R}}}^{\lambda, h}\right) }^2 = {\lambda}^2 U^2 -  {h}^2, \\
    {\left( d{\Lambda}^{\lambda, h} \right)}^2
    = \frac{ {\lambda}^2 U ^2 }{\left( {\lambda}^2 U^2 -  {h}^2  \right)^2}
    \left( {\lambda}^2 U^2 \, \left( 1 - {\lambda}^2 \left(  \frac{ dU   }{ ds } \right)^2   \right) - h^2
    \right) \, ds^2, \\
    {d {\Theta}^{\lambda, h}} = \frac{h}{{\lambda}^2 U^2  } d{\Lambda}^{\lambda, h}.
\end{cases}
\end{equation}
 Then, the helicoidal surface ${\Sigma}^{\lambda, h}$ is isometric to the initial surface ${\Sigma}$, and
 its angle function ${n}^{\lambda, h}_{3}={\mathbf{n}}^{\lambda, h} \cdot \mathbf{k}$ is
 determined by the Bour function $U$ of the initial surface $\Sigma$.
 \begin{equation*}
 {\left({n}^{\lambda, h}_{3}\right)}^2 =   {\lambda}^2 \left(  \frac{dU}{ds}
 \right)^2.
\end{equation*}
 \textbf{(C)} Furthermore, the helicoidal surface ${\Sigma}^{1, \mu}$ coincides with the initial surface ${\Sigma}$.
\end{lemma}

\begin{proof}
 \textbf{(A)} The definitions of the Bour coordinate $(s,t)$ and the Bour function $U$ yield
 \begin{eqnarray*}
  I_{\Sigma} &=& \left(   dR^2 +  {d{\Lambda}}^2 \right) + 2 \mu  d{\Lambda} d\theta +
   \left(   {\mathcal{R}}^2 +  {\mu}^2 \right) d  {\theta}^2 \\
  &=& \left(   d{\mathcal{R}}^2 +  \frac{{\mathcal{R}}^2}{{\mathcal{R}}^2 + {\mu}^2 } {d{\Lambda}}^2 \right)
  +  \left(   {\mathcal{R}}^2 +  {\mu}^2 \right) \left( d  {\theta} +   \frac{\mu}{{\mathcal{R}}^2 + {\mu}^2 } d\Lambda  \right)^2  \\
  &=&  ds^2 + U^2 dt^2 .
 \end{eqnarray*}
 Noticing that the definition $U^2 =R^2 +  {\mu}^2$ implies $d{\mathcal{R}}^2 = \frac{U^2}{U^2 -
 {\mu}^{2}} dU^2$, we can recover the function $\dot{\Lambda}=\frac{ d\Lambda}{ ds }$ from the Bour function $U(s)$ explicitly:
 \begin{equation*}
 ds^2 = d{\mathcal{R}}^2 + \frac{{\mathcal{R}}^2}{{\mathcal{R}}^2 + {\mu}^2 } d\Lambda^2
 = \frac{U^2}{U^2 - {\mu}^{2}} dU^2 + \frac{U^2 - {\mu}^{2}}{U^2} d\Lambda^{2},
\end{equation*}
 and
 \begin{equation*}
 d\Lambda^2 = \frac{U^2}{U^2 - {\mu}^{2}} \left( ds^2 -\frac{U^2}{U^2 - {\mu}^{2}}  dU^2 \right)
 = \frac{ U ^2 }{\left( U^2 -  {\mu}^2  \right)^2} \left(  U^2 \, \left( 1 - \left( \frac{
 dU}{ ds } \right)^2  \right)  - h^2  \right) \, ds^{2}.
\end{equation*}
Adopting the symbol $\,\dot{}=\frac{d}{ds}$ again, we obtain
\begin{equation*}
  {\mathbf{X}}_{s} \times {\mathbf{X}}_{t} =  \left(  \, \mu \dot{\mathcal{R}}  \sin  \theta -\mathcal{R} \dot{\Lambda}   \cos
 \theta, \,  - \mu \dot{\mathcal{R}}  \cos  \theta - \mathcal{R} \dot{\Lambda}   \sin  \theta, \,\mathcal{R} \dot{\mathcal{R}} \,
 \right).
\end{equation*}
After setting $I_{\Sigma} := E ds^2 +2F ds dt + G dt^2  = ds^2 +
U^2 dt^2$, we immediately see that
 \begin{equation*}
  { \Vert {\mathbf{X}}_{s} \times {\mathbf{X}}_{t} \Vert}^2 = E G - F^2= U^{2}.
 \end{equation*}
It thus follows that
\begin{equation*}
 {n_3}^2 = \frac{{ \left( \mathcal{R} \dot{\mathcal{R}} \right) }^2}{U^2} =  { \dot{U} }^2 = \left(
 \frac{dU}{ds} \right)^{2}.
\end{equation*}
\textbf{(B)} We first show that the surface ${\Sigma}^{\lambda,
h}$ is isometric to the initial
 surface ${\Sigma}$. Let us write
 \begin{equation*}
 I_{{\Sigma}^{\lambda, h}} = \, E^{\lambda, h} ds^2
 +  \, 2 F^{\lambda, h} ds dt +  \, G^{\lambda, h} dt^2 .
 \end{equation*}
 Adopting the symbol $\,\dot{}=\frac{d}{ds}$ and using (\ref{data2}), we have
 \begin{eqnarray*}
   E^{\lambda, h}
  &=&  {\left( {\dot{\mathcal{R}}}^{\lambda, h}\right)   }^2 + {\mathcal{R}}^2 {\left(  {\dot{\Theta}}^{\lambda, h}\right)}^2
   + {\left(  {\dot{\Lambda}}^{\lambda, h}   - h {\dot{\Theta}}^{\lambda, h} \right)}^{2}  \\
  &=&  {\left( {\dot{\mathcal{R}}}^{\lambda, h}\right)   }^2 +
  \frac{  {\lambda}^2 U^2 - h^2 }{  {\lambda}^2 U^2  }   {\left(  {\dot{\Theta}}^{\lambda, h}\right)}^2   \\
  &=&   \frac{  {\lambda}^2 U^2 { \dot{U}  }^2 }{  {\lambda}^2 U^2 - h^2 }
  +  \frac{  {\lambda}^2 U^2 - h^2 }{  {\lambda}^2 U^2  }  \cdot \frac{  {\lambda}^2 U^2
  \left[ {\lambda}^2 U^2 \, \left( 1 - {\lambda}^2  { \dot{U}  }^2 \right) - h^2
   \right] }{  {\left(   {\lambda}^2 U^2 - h^2\right)   }^2 }      \\
  &=& 1.
 \end{eqnarray*}
 We also deduce
 \begin{equation*}
   F^{\lambda, h} = - \frac{1}{\lambda} \left[ \, \left( {\left( {\mathcal{R}}^{\lambda, h} \right)}^{2} +
   h^2 \right) {\dot{\Theta}} - h {\dot{\Lambda}} \right] =
    - \frac{1}{\lambda} \left[ \,  {\lambda}^2 U^2 {\dot{\Theta}} - h {\dot{\Lambda}}   \right] = 0,
 \end{equation*}
and
 \begin{equation*}
    G^{\lambda, h} =  \frac{1}{{\lambda}^{2}} \left[ {\left( {\mathcal{R}}^{\lambda,
   h}\right)}^2 + h^2 \right]= U^2 .
 \end{equation*}
 Combining these, we  meet
    \begin{equation*}
I_{{\Sigma}^{\lambda, h}}=  \, E^{\lambda, h} ds^2
 +  \, 2 F^{\lambda, h} ds dt +  \, G^{\lambda, h} dt^2 = ds^2 + U^2 dt^2 =
 I_{\Sigma}.
 \end{equation*}
 Now, it remains to determine the angle function of the surface ${\Sigma}^{\lambda, h}$.
 Adopting the new variable $\theta= \frac{t}{\lambda} - {\Theta}^{\lambda, h}$ for simplicity, we
 write
\begin{equation*}
  {{\mathbf{X}}^{\lambda, h}}_{s} \times {{\mathbf{X}}^{\lambda, h}}_{t} =
  \frac{1}{\lambda} \left(  \, h {\dot{\mathcal{R}}}^{\lambda, h}  \sin  \theta - {\mathcal{R}}^{\lambda, h} \dot{\Lambda}   \cos
 \theta, \,  - h {\dot{\mathcal{R}}}^{\lambda, h}  \cos  \theta - {\mathcal{R}}^{\lambda, h} \dot{\Lambda}   \sin  \theta, \,
 {\mathcal{R}}^{\lambda, h} {\dot{\mathcal{R}}}^{\lambda, h} \, \right).
\end{equation*}
 Taking account into this and the equality
\begin{equation*}
 { \Vert {{\mathbf{X}}^{\lambda, h}}_{s} \times {{\mathbf{X}}^{\lambda, h}}_{t}
 \Vert}^2 =  E^{\lambda, h}  G^{\lambda, h} - {\left(   F^{\lambda, h} \right)}^2 = U^{2},
\end{equation*}
  we meet
 \begin{equation*}
 {\left({n}^{\lambda, h}_{3}\right)}^2 =  {\left( {\mathbf{n}}^{\lambda, h} \cdot \mathbf{k} \right)}^2 =
   \frac{1}{U^2} \cdot \frac{  {\left(  {\mathcal{R}}^{\lambda, h} \right)}^2
  {\left(  {\dot{\mathcal{R}}}^{\lambda, h} \right)}^2 }{ {\lambda}^2}
  = {\lambda}^2  { \dot{U}  }^2  = {\lambda}^2 \left(  \frac{dU}{ds} \right)^2.
 \end{equation*}
 \textbf{(C)} The datum $\left({\mathcal{R}}^{1, \mu}, {\Lambda}^{1, \mu}, {\Theta}^{1, \mu}  \right)$
 of ${\Sigma}^{1, \mu}$ coincides with the datum $\left({\mathcal{R}}, {\Lambda}, {\Theta} \right)$ of ${\Sigma}$.
\end{proof}

 We briefly sketch the geometric ingredients in our construction in Theorem \ref{moduli0}.
 For given a helicoidal ${\mathcal{K}}^{\frac{1}{4}}$-translator, we prove that there exists a sub-family chosen from the two-parameter family
 of Bour's isometric helicoidal surfaces, so that each member of this sub-family
 is a ${\mathcal{K}}^{\frac{1}{4}}$-translator and that one member is rotationally symmetric.

 Our one-parameter family of ${\mathcal{K}}^{\frac{1}{4}}$-translators admits the parametrizations by so called the Bour coordiate
 $(s,t)$ and the Bour function $U=U(s)$. The trick to obtain the explicit construction in \textbf{(C3)} is
  to perform the coordinate transformation $s \mapsto U$ to have the geometric coordinate $(U,t)$ on our one-parameter
 family of  ${\mathcal{K}}^{\frac{1}{4}}$-translators.

\begin{lemma}[\textbf{Existence of helicoidal ${\mathcal{K}}^{\frac{1}{4}}$-translators of pitch $h$}]
\label{hel1}
 Let $h$ be a given constant. Then, any non-cylindrical helicoidal ${\mathcal{K}}^{\frac{1}{4}}$-translator
 with pitch $h$ admits the parametrization
\begin{equation*}
    \left( U, t \right)
  \mapsto \left(\; {\mathcal{R}(U)}  \cos  \left(\, t - {\Theta}(U) \, \right), \; {\mathcal{R}(U)} \sin \left(
  \, t -  {\Theta}(U) \, \right), \; {\Lambda}(U)  + h \left( \, t - {\Theta}(U) \,  \right) \; \right),
\end{equation*}
 where the geometric datum $\left( {\mathcal{R}}(U), {\Lambda}(U), {\Theta}(U) \right)$
 can be obtained from the relation
\begin{equation} \label{data3}
\begin{cases}
    {\mathcal{R}(U)}^2 =    U^2 -  {h}^2 , \\
    {\left(  \; \frac{d  {\Lambda}}{dU}  \; \right)}^2
    = \frac{  U^2  }{ {\left(  U^2 -  {h}^2  \right)}^2 }
     \left[  U^4 + \left( c - 1 - {h}^2  \right) U^2  - h^2 c  \, \right], \\
    {\left(  \; \frac{d  {\Theta}  }{dU}   \; \right)}^2
    = \frac{  {h}^2  }{  U^2  {\left(  U^2 -  {h}^2  \right)}^2 }
    \left[  U^4 + \left( c - 1 - {h}^2  \right) U^2  - h^2 c \,
    \right],
\end{cases}
\end{equation}
 where $c  \in \mathbb{R}$ is a  constant.
\end{lemma}

\begin{proof}
 Taking $\lambda=1$ in Lemma \ref{B00}, we construct a helicoidal
 surface $\Sigma$ with pitch $h$:
\begin{equation*}
    \left( s, t \right) \mapsto {\mathbf{X}}^{1, h} \left( s, t \right)
   = \left(\; {\mathcal{R}}  \cos  \left(\, t - {\Theta}  \, \right), \; {\mathcal{R} } \sin \left(
  \, t -  {\Theta}  \, \right), \; {\Lambda}   + h \left( \, t - {\Theta}  \,  \right) \; \right),
\end{equation*}
 where the geometric datum $\left( {\mathcal{R}}, {\Lambda}, {\Theta}  \right)=
 \left( {\mathcal{R}}(s), {\Lambda}(s), {\Theta}(s)  \right)$
 is given by the relation
\begin{equation} \label{data3a}
\begin{cases}
 { {{\mathcal{R}}}  }^2 =   U^2 -  {h}^2, \\
    {\left( d{\Lambda}  \right)}^2
    = \frac{   U ^2 }{\left(   U^2 -  {h}^2  \right)^2}
    \left(   U^2 \, \left( 1 -  \left(  \frac{ dU   }{ ds } \right)^2   \right) - h^2
    \right) \, ds^2, \\
    {d {\Theta} } = \frac{h}{  U^2  } d{\Lambda}.
\end{cases}
\end{equation}
 The key point is to take the Bour function $U$ as the new parameter on our helicoidal surface $\Sigma$.
  According to Lemma \ref{B00} again, we see that the induced metric on  $\Sigma$ reads
 $I_{\Sigma} = ds^2 + U^2 dt^{2}$, that its Gaussian curvature  $K$ is equal to
 $K=- \frac{1}{U} \frac{d^2 U}{ds^2}$, and that its angle function reads
 ${n_3}^2 =  \left(  \frac{dU}{ds} \right)^{2}$. Thus, the condition that the helicoidal surface $\Sigma$ becomes a ${\mathcal{K}}^{\frac{1}{4}}$-translator
 implies that $K =  {n_3}^4$, which means the ordinary differential equation
\begin{equation*}
 - \frac{1}{U} \frac{d^2 U}{ds^2} =  \left(  \frac{dU}{ds} \right)^4.
\end{equation*}
 In the case when $\frac{dU}{ds}$ vanishes locally, our surface $\Sigma$ becomes the cylinder over
 a circle in the $xy$-plane. When  $\frac{dU}{ds}$ does not vanish, we are able to make a coordinate
 transformation $s \mapsto U$ and can rewrite the above ODE as
 \begin{equation*}
  0 = \frac{d }{ds} \left( \frac{1}{{\left( \frac{dU}{ds}\right)}^2} - U^2 \right).
 \end{equation*}
  Hence its first integral is explicitly given by, for some constant $c \in \mathbb{R}$,
\begin{equation*}
 {ds^2} =  \left( U^{2} + c  \right)  dU^{2}.
\end{equation*}
 We now can employ this to perform the coordinate transformation $\left(s, t \right) \mapsto   \left( U, t \right)$
 on $\Sigma$. Rewriting (\ref{data3a}) in terms of the new variable $U$ gives indeed the relation in (\ref{data3}).
\end{proof}

 \begin{proof}[\textbf{Proof of  Theorem \ref{moduli0}}] We first prove \textbf{(B)}. Taking $h=0$ in Lemma
 \ref{hel1}, we see that any rotational ${\mathcal{K}}^{\frac{1}{4}}$-translators admits the patch
\begin{equation*} 
    \left( U, t \right)
  \mapsto \left(\; {\mathcal{R}(U)}  \cos  \left(\, t - {\Theta}(U) \, \right), \; {\mathcal{R}(U)} \sin \left(
  \, t -  {\Theta}(U) \, \right), \; {\Lambda}(U)  + h \left( \, t - {\Theta}(U) \,  \right) \; \right),
\end{equation*}
 where the geometric datum $\left( {\mathcal{R}}(U), {\Lambda}(U), {\Theta}(U) \right)$
 satisfies the relation
\begin{equation*} \label{data4}
      {\left(  {\mathcal{R}}(U)  \right)}^2   =   U^2 , \quad
      {\left( \frac{d  {\Lambda}}{dU}  \right)}^2   =   U^2 + \left( c - 1  \right),  \quad
       {\left(  \frac{d  {\Theta}  }{dU}  \right)}^2 =  0
\end{equation*}
 for some constant $c \in \mathbb{R}$. The condition that the helicoidal surface $\Sigma$ becomes
 a ${\mathcal{K}}^{\frac{1}{4}}$-translator implies the ordinary differential equation
 \begin{equation*}
 - \frac{1}{U} \frac{d^2 U}{ds^2} =  \left( \frac{dU}{ds} \right)^{4}.
 \end{equation*}

 When $\frac{dU}{ds}$ vanishes locally, our surface $\Sigma$ becomes the cylinder over
 a circle in the $xy$-plane. In the case when $\frac{dU}{ds}$ does not vanish, we can introduce a coordinate
 transformation $s \mapsto U$.  Since $ \frac{d  {\Theta}  }{dU}$ vanishes,
without loss of generality, after a translation of the coordinate
$t$, we may take ${\Theta} = 0$ in the above patch as follows
\begin{equation*}
    \left( U, t \right)
  \mapsto \left(\; U   \cos  t, \;  U \sin t, \; {\Lambda}(U)   \;
  \right).
\end{equation*}
 As in the proof of Lemma \ref{hel1}, ${\Lambda}(U)$ solves the ordinary differential equation
 \begin{equation*}
 \frac{d {\Lambda}}{dU}   = \pm \sqrt{U^2 + \left( c - 1 \right)}.
 \end{equation*}
 Considering the sign of the constant $c-1$,  we meet the explicit solution ${\Lambda}_{c}(U)={\Lambda}(U)$ (up to the sign) as follows.
  \begin{equation*} \label{rot}
  {\Lambda}(U) =
\begin{cases}
   \frac{1}{2} \left[ \, U \sqrt{U^2 + {\kappa}^{2}\,}  + {\kappa}^{2} \, \mathrm{arcsinh} \,
 \left( \frac{U}{{\kappa}}  \right)  \; \right]  \quad \quad (\text{when}\; c=1 + {\kappa}^{2}, \; \kappa>0),  \\
 \frac{1}{2} \; U^2 \quad \quad \quad \quad  \quad \quad  \quad \quad \quad \quad  \quad \quad \quad \quad
   \quad   (\text{when}\; c=1),  \\
   \frac{1}{2} \left[ \,
 U \sqrt{U^2 - {\kappa}^{2}\,} -  {\kappa}^{2} \, \mathrm{arccosh} \,  \left( \frac{U}{{\kappa}}  \right)
  \; \right]   \quad \quad (\text{when}\; c=1 - {\kappa}^{2}, \; \kappa>0).
\end{cases}
\end{equation*}

  We next prove \textbf{(A)}. Using Lemma \ref{B00}, we see that, for a given
  helicoidal ${\mathcal{K}}^{\frac{1}{4}}$-translator
 $\Sigma$, we are able to introduce the Bour coordinate $(s,t)$ and the Bour function
 $U(s)$ on the surface $\Sigma$ so that $I_{\Sigma} = ds^2 + {U(s)}^2 dt^2$. The condition that
 $\Sigma$ is a ${\mathcal{K}}^{\frac{1}{4}}$-translator says
\begin{equation} \label{angle}
 - \frac{1}{U} \frac{d^2 U}{ds^2} =  \left(  \frac{dU}{ds} \right)^4,
\end{equation}
 just as we saw in the proof of Lemma \ref{hel1}. Next, by Lemma \ref{B00} again, we can associate a one-parameter family of isometric helicoidal surfaces
 ${\Sigma}^{h}$ satisfying that $I_{{\Sigma}^{h}} =I_{\Sigma} $, that ${\Sigma} =
 {\Sigma}^{\mu}$, and that the angle function on ${\Sigma}^{h}$ coincide with the one on ${\Sigma}$.
 Hence, as we saw in the proof of Lemma \ref{hel1}, the above ordinary
 differential equation in (\ref{angle}) guarantees that any
 helicoidal surface ${\Sigma}^{h}$ becomes indeed a ${\mathcal{K}}^{\frac{1}{4}}$-translator.

 It now remains to show \textbf{(C)}. The statement \textbf{(C1)} is obvious by the construction in Lemma \ref{hel1}.
 Next, the equality ${ds^2} =  \left( U^{2} + c  \right) dU^{2}$ proved in Lemma \ref{hel1} implies that the
 induced metric of the helicoidal surface constructed in Lemma \ref{hel1} reads
 \begin{equation*}
 ds^2 + U^2 dt^2 = \left( U^{2} + c  \right)  dU^2 +  U^2  dt^{2},
 \end{equation*}
 (which implies \textbf{(C2)} and \textbf{(C3)}), and that the angle function is given by, up to a sign,
  \begin{equation*}
  \frac{dU}{ds} = \frac{1}{ {  \frac{ds}{dU} } } =  \frac{1}{\sqrt{U^2 +c}},
 \end{equation*}
 which is \textbf{(C4)}.
 This complete the proof of our description of the moduli space of helicoidal ${\mathcal{K}}^{\frac{1}{4}}$-translators
 in Theorem \ref{moduli0}.
\end{proof}

\end{document}